\documentclass[reqno]{amsart}

\usepackage[utf8]{inputenc}
\usepackage{times}
\usepackage[UKenglish]{babel}
\usepackage[T1]{fontenc}
\usepackage{stmaryrd,verbatim,calligra,mathrsfs}
\usepackage[normalem]{ulem}

\usepackage[a4paper,top
=2cm,bottom=2cm,left=2cm,right=2cm,marginparwidth=1.75cm]{geometry}

\usepackage{amsmath,amssymb,amsfonts,amsthm,upgreek}
\usepackage{mathrsfs, mathtools}
\usepackage{graphicx}
\usepackage{tikz-cd}
\usepackage[all]{xy}
\usepackage{cancel}
\usepackage{appendix}
\usepackage{enumerate}
\usepackage{multicol}
\usetikzlibrary{calc}

\usepackage[colorlinks=true, allcolors=blue,backref=page]{hyperref}
\usepackage[initials,alphabetic]{amsrefs}

\newcommand{\bN}{{\bf N}}

\newcommand{\bP}{{\bf P}}

\newcommand{\bZ}{{\bf Z}}

\newcommand{\sE}{\mathscr{E}}
\newcommand{\sF}{\mathscr{F}}
\newcommand{\sG}{\mathscr{G}}

\newcommand{\sI}{\mathscr{I}}

\newcommand{\sL}{\mathscr{L}}

\newcommand{\sO}{\mathscr{O}}

\newcommand{\sQ}{\mathscr{Q}}

\newcommand{\N}{\bN}
\newcommand{\Z}{\bZ}

\newcommand{\R}{\mathbb{R}}
\newcommand{\C}{\mathbb{C}}

\newcommand{\Pic}{{\rm Pic}}
\renewcommand{\P}{\bP}

\renewcommand{\epsilon}{\varepsilon}

\newcommand{\Hom}{{\mathrm{Hom}}}
\DeclareMathOperator{\chom}{\mathscr{H}\text{\kern -3pt {\calligra\large om}}\,}

\newcommand{\ch}{{\rm ch}}

\newcommand{\id}{\mathrm{id}}

\renewcommand{\Im}{\mathop{\mathrm{Im}}}

\renewcommand{\Re}{\mathop{\mathrm{Re}}}

\newcommand{\tr}{\mathop{\mathrm{tr}}\nolimits}
\newcommand{\rk}{\mathop{\mathrm{rk}}\nolimits}

\newcommand{\qandq}{\quad\text{and}\quad}
\newcommand{\qwithq}{\quad\text{with}\quad}
\newcommand{\qforq}{\quad\text{for}\quad}
\newcommand{\qforallq}{\quad\text{for all}\quad}

\newcommand{\ses}[3]{#1 \hookrightarrow #2 \twoheadrightarrow #3}

\def\<{\mathopen{}\left<}
\def\>{\right>\mathclose{}}
\def\({\mathopen{}\left(}
\def\){\right)\mathclose{}}

\usepackage{multicol, color}

\definecolor{gold}{rgb}{0.85,.66,0}
\definecolor{cherry}{rgb}{0.9,.1,.2}
\definecolor{burgundy}{rgb}{0.8,.2,.2}
\definecolor{orangered}{rgb}{0.85,.3,0}
\definecolor{orange}{rgb}{0.85,.4,0}
\definecolor{olive}{rgb}{.45,.4,0}
\definecolor{lime}{rgb}{.6,.9,0}
\definecolor{green}{rgb}{.2,.7,0}
\definecolor{grey}{rgb}{.4,.4,.2}
\definecolor{brown}{rgb}{.4,.3,.1}

\theoremstyle{plain}
\newtheorem{theorem}{Theorem}[section]
\newtheorem*{theorem*}{Theorem}
\newtheorem{proposition}[theorem]{Proposition}
\newtheorem{lemma}[theorem]{Lemma}

\theoremstyle{definition}
\newtheorem{definition}[theorem]{Definition}

\theoremstyle{remark}
\newtheorem{remark}[theorem]{Remark}

\numberwithin{equation}{section}

\title{Asymptotically Z-stable bundles over projective surfaces}
\author{Luiz Lara}
\address[Luiz Lara]{Instituto de Matm\'{a}tica, Estat\'{i}stica e Computa\c{c}\~{a}o Cient\'{i}fica (IMECC), Universidade Estadual de Campinas (UNICAMP), Campinas - SP, Brazil}
\email{\href{luizlara@ime.unicamp.br}{luizlara@ime.unicamp.br}}
\author{Henrique N. Sá Earp}
\address[Henrique N. Sá Earp]{Instituto de Matm\'{a}tica, Estat\'{i}stica e Computa\c{c}\~{a}o Cient\'{i}fica (IMECC), Universidade Estadual de Campinas (UNICAMP), Campinas - SP, Brazil}
\email{\href{henrique.saearp@ime.unicamp.br}{henrique.saearp@ime.unicamp.br}}
\date{\today}

\begin{document}

\begin{abstract}
    We study the existence of asymptotically $Z$-stable (a.Z stable) bundles over polycyclic surfaces. Our choice of polynomial central charge is related to the existence of solutions of the deformed Hermitian--Yang--Mills equations, with vanishing $B$-field, in the large-volume limit. The main result is a technique to construct rank $3$, strictly a.Z-stable bundles as extensions of a line bundle by a $\mu$-stable bundle of rank $2$. In particular, this leads to new examples of strictly a.Z-stable bundles over $\P^2$, the product $\P^1\times \P^1$, and the blow-up $\mathrm{Bl}_q\P^2$. We also present an analogue of the Hoppe criterion for the a.Z-stability of vector bundles of rank $2$, which may be of independent interest.
\end{abstract}
\maketitle

\tableofcontents

\section{Introduction}
\setcounter{theorem}{0}
\renewcommand{\thetheorem}{\arabic{theorem}}
The deformed Hermitian--Yang--Mills (dHYM) equations first emerged in the context of String Theory, where they appeared as the so-called \textit{Fourier--Mukai} transform of the Special Lagrangian equations for submanifolds of a Calabi-Yau manifold $X$. The first version of the equation, which is most prevalent in the literature, was stated in terms of the curvature of a line bundle $\sL\to X$, see eg.~\cites{MarcosMarino_2000,leung_special_2000,madnick_spn-instantons_2024}. A more general version of the dHYM equation, in terms of the curvature of a higher-rank holomorphic vector bundle $\sE\to X$ 
over an arbitrary compact Kähler manifold $(X,g,J,\omega)$, was proposed by Collins and Yau in~\cite{collins_moment_2018}, and it can be stated as follows. 
Let $A$ be a connection on $\sE$; given an integer $k$, we say that $A$ is a \emph{$k$-dHYM connection} if
\begin{align}
\label{eq:dhym-1}
    \Im\left(e^{-i\varphi_k(\sE)}\left(k\omega\otimes\id_\sE - \frac{F_A}{2\pi}\right)^n\right) = 0,
\end{align}
where 
\begin{align}
    \varphi_k(\sE) = \arg\left(\int_X \tr\left(k\omega\otimes\id_\sE - \frac{F_A}{2\pi}\right)^n\right).
\end{align}

In the rank $1$ case, the existence of solutions to the $k$-dHYM equations has been studied in a series of works~\cites{collins_deformed_2017, collins_moment_2018, collins_stability_2022}, from the analytical point of view, by way of techniques from Geometric Invariant Theory. On the other hand, from the perspective of Algebraic Geometry, important progress was made in~\cites{chen_j-equation_2021, chen_supercritical_2021}, where solutions to the dHYM equations are identified with solutions of the so-called $J$-equations, and their existence is shown to be related to the $J$-stability of the underlying vector bundle. In the higher-rank case, some recent work by Éder Correa~\cites{correa2023deformed, correa_dhym_2024} undertakes a systematic study of solutions to the dHYM equations on rank $2$ \emph{decomposable} vector bundles over flag varieties. In particular, examples of \emph{strict} dHYM connections  --- that is, connections satisfying the dHYM equations while failing to be classically HYM --- were constructed, highlighting the subtlety of the dHYM condition. 

Further insights into the existence problem in higher rank were obtained in~\cite{dervan_z-critical_2024}, where the authors considered a more general class of equations, known as $Z$-critical equations, encompassing the dHYM equations as a special case. It was shown that, in the so-called large-volume limit, the existence of solutions is equivalent to a new stability condition termed \emph{asymptotic $Z$-stability}, thus providing a bridge between analytical gauge theory and algebro-geometric stability conditions.
Soon after, in~\cite{keller_z-critical_2024}, the focus shifted to the study of $Z$-critical equations on rank $2$ vector bundles over complex surfaces, beyond the large-volume regime. In this setting, (non-asymptotic) $Z$-stability becomes a necessary condition for the existence of solutions, underscoring the delicate interplay between curvature, stability, and complex geometry in the theory of $Z$-critical connections.

The concept of \emph{asymptotic $Z$-stability} will be central to our study. Let $(X,L)$ be a polarized 
variety, and let $\sE\to X$ be a holomorphic vector bundle. For each $k\in\N$, we define the \emph{$k$-central charge} of $\sE$ by
\begin{align}
\label{eq:k-central-charge}
    Z_k(\sE)\coloneqq 
    i^{n+1} \ch(\sE)\cdot e^{-ikL}\in\C,
\end{align}
where $\ch(\sE)\in H^\ast(X,\R)$ denotes the total Chern character of $\sE$, and the \emph{$k$-slope} of $\sE$ by
\begin{equation}\label{eq:def-k-slope}
    \mu_{k}(\sE) \coloneqq 
    -\frac{\Re(Z_k(\sE))}{\Im(Z_k(\sE))}.
\end{equation}
We say that $\sE$ is \emph{asymptotically $Z$-stable} (a.Z-stable for short) if, for every proper coherent subsheaf $\sF\hookrightarrow\sE$ with torsion-free quotient, there exists $k_0>0$ such that  
\begin{equation}
\label{eq:asym-z-est}
    \mu_k(\sF) \leq \mu_k(\sE), \quad \forall k\geq k_0.
\end{equation}
Such stability guarantees the existence of solutions to ~\eqref{eq:dhym-1} for all sufficiently large $k$, cf. \cite{dervan_z-critical_2024}.
By contrast, recall that the existence of HYM connections is related to Mumford-Takemoto stability, also called  $\mu$-stability, which is defined in terms of the slope 
\begin{equation}
    \mu(\sE) \coloneqq \frac{c_1(\sE)\cdot L^{n-1}}{\rk(\sE)},
\end{equation}
where similarly $\sE$ is called \emph{$\mu$-semistable} if, for every proper coherent subsheaf $\sF\hookrightarrow\sE$,
\begin{equation}
    \mu(\sF) \leq \mu(\sE),
\end{equation}
and when the inequality holds strictly, we say that $\sE$ is \emph{$\mu$-stable}. 

We know from~\cite{dervan_z-critical_2024}  that every a.Z-stable vector bundle is automatically $\mu$-semistable, and it is natural to ask whether there exist a.Z-stable bundles which are not $\mu$-stable, to which we will refer as \textit{strictly} a.Z-stable. This question is also motivated from the analytic point of view, since each of these stability conditions corresponds to the existence of connections on $\sE$ solving the (d)HYM equations (respectively). More precisely, on a strictly a.Z-stable bundle, there exist solutions to the dHYM equation, but not to the classical HYM equation, which suggests that dHYM gauge theory is a genuinely new feature of the polarized variety $(X,L)$. A first nontrivial construction of such bundles is presented in the recent work~\cite{correa_dhym_2024}, obtaining rank $2$ decomposable vector bundles over the full flag variety $\P(T\P^2)$. As natural next steps, we therefore set out to construct  strictly a.Z-stable bundles which are \emph{indecomposable} and have \emph{higher rank}. To the best of our knowledge, these are the first such examples in the literature.

\subsection{Outline and statement of main results}

We start \S\ref{sec:basic-facts} by defining in detail the central charge and $k$-slope for a coherent sheaf $\sE$ over a polarized surface $(X,L)$.
As a consequence we obtain an interesting relation between a.Z-stability and Gieseker stability, when $X$ is a Del Pezzo surface, cf.  Proposition~\ref{prop:az-stability-gieseker}. 
In \S\ref{sec:adapted-hoppe-criteria}, we address cohomological stability criteria: we recall the Hoppe-type criterion for $\mu$-stability over polarized varieties, with special attention to Hartshorne-Serre bundles, restating~\cite{Jardim2017}*{Proposition~13} in our context and notation. 
Finally, we present a Hoppe-type criterion for the a.Z-stability of rank $2$ vector bundles in Proposition~\ref{prop:asymp-hopp}, along the lines of ~\cite{okonek_vector_1980}*{Lemma~1.2.5}, in which some readers may find independent interest. 

Having these preparations in place, \S\ref{sec:const-az-stab} is dedicated to proving our main theorems. We begin by proving Proposition~\ref{teo:strict-z-stable-bun}, which gives a way to construct strictly a.Z-stable bundles, starting from suitable $\mu$-stable vector bundles of rank $2$. This is a version of Maruyama's example~\cite{okonek_vector_1980}*{Pg.~90}, also mentioned in~\cites{dervan_z-critical_2024,keller_z-critical_2024}, in the context of $Z$-stability for other central charges.  
Specializing to the case where our rank $2$ bundles are Hartshorne-Serre, we combine Proposition~\ref{teo:strict-z-stable-bun} and Proposition~\ref{prop:HS-stability} to obtain Theorem~\ref{teo:hom-cond-for-stric-z-stab} below, a recipe for the construction of strictly a.Z-stable bundles over a polarized surface $(X,L)$, in terms of a pair $(Z,D)$, where $Z\subset X$ is a local complete intersection subscheme of points and $D\in\Pic(X)$ is a divisor. Throughout this text, we denote the \emph{length} of $Z$ by $\ell(Z) := \dim H^0(X,\sO_Z)$.

\begin{theorem}
\label{teo:hom-cond-for-stric-z-stab}
    On a polarized polycyclic surface $(X,L)$, let $R_0\subset \Pic(X)$ be a region  such that \[H^0(X,\sO_X(B)) = 0, \qforallq B\notin R_0.\] 
    For each pair $(Z,D)$, where $Z\subset X$ is a local complete intersection subscheme of points and $D\in\Pic(X)$, we define the region 
    \begin{equation}
        R \coloneqq \{B\in R_0 \;|\;B\cdot L \leq D\cdot L\}.
    \end{equation}
    Assume that a given pair $(Z,D)$ satisfies the following conditions:
    \begin{enumerate}[(a)]
        \item\label{cond:1} $D\cdot L > 0$;
        \item\label{cond:2} $D^2 < \ell(Z)$;
        \item\label{cond:3} $H^0(X,\sI_Z(B)) = 0$, for all $B\in R$;
        \item\label{cond:4} either $H^1(X,\sO_X(-D)) \neq 0$, or $H^2(X,\sO_X(-D)) = 0$ and $\ell(Z) > 2\chi(\sO_X)$;
        \item\label{cond:5} $H^2(X,\sO_X(-2D)) = 0$.
    \end{enumerate}
    Then, the rank $3$ sheaf $\sF_{Z,D}$ obtained as the double extension 

    \begin{equation}\label{eq:double-extension}
        \begin{tikzcd}[row sep=0.4cm, column sep=0.4cm]
            & \sO_X \arrow[hook]{r} & \sE_{Z,D} \arrow[two heads]{r} \arrow[hook]{d}& \sI_Z(2D) \\
            & & \sF_{Z,D} \arrow[two heads]{d} \\
            & & \sO_X(D)    
        \end{tikzcd}        
    \end{equation}
    is an indecomposable, holomorphic, strictly a.Z-stable vector bundle.
\end{theorem}
\begin{remark} The interpretation of each condition in the theorem above can be seen in the following picture.  
\begin{figure}[h!]
    \centering
    \resizebox{\textwidth}{!}{%
\begin{tikzpicture}
\tikzstyle{every node}=[font=\large]

\node at (6.5,16.5) {\small\textbf{Cohomological condition}};
\node at (14.75,16.5) {\small  \textbf{Geometric interpretation}};

\node at (7.25,15.25) {\small $D \cdot L > 0$};
\node at (7.25,14) {\small $D^2 < \ell(Z)$};
\node at (6.25,12.75) {\small $H^0(X,\sI_Z(B)) = 0, \;\forall B\in R$};
\node at (5,11.5) {\small $H^1(X,\sO_X(D)) \neq 0$, or};
\node at (5,10.75) {\small $H^2(X,\sO_X(-D)) = 0$, with $\ell(Z) > 2\chi(X)$};
\node at (6.5,9.5) {\small $H^2(X,\sO_X(-2D)) = 0$};

\draw (5.75,15.75) rectangle (8.75,14.75); 
\draw (5.75,14.5) rectangle (8.75,13.5); 
\draw (3.8,13.25) rectangle (8.75,12.25); 
\draw (1.15,12) rectangle (8.75,10.25); 
\draw (4.25,10) rectangle (8.75,9); 

\node at (14,15.25) {\small $\mu(\sE) > 0$};
\node at (15,14) {\small $\sE$ does not a.Z-destabilize $\sF$};
\node at (15,12.75) {\small $\sE$ is $\mu$-stable};
\node at (15,11.25) {\small $\sF$ is a non-trivial extension};
\node at (16.25,9.5) {\small $\sE$ is a locally free non-trivial extension};

\draw (12.5,15.75) rectangle (15.5,14.75); 
\draw (12.5,14.5) rectangle (17.5,13.5); 
\draw (12.5,13.25) rectangle (17.5,12.25); 
\draw (12.5,12) rectangle (17.5,10.25);
\draw (12.5,10) rectangle (20,9);

\draw[<->, >=Stealth] (9.5,15.25) -- (12,15.25);
\draw[<->, >=Stealth] (9.5,14) -- (12,14);
\draw[<->, >=Stealth] (9.5,12.75) -- (12,12.75);
\draw[<->, >=Stealth] (9.5,11.25) -- (12,11.25);
\draw[<->, >=Stealth] (9.5,9.5) -- (12,9.5);

\end{tikzpicture}
}
    \caption{Geometric interpretation of Theorem~\ref{teo:hom-cond-for-stric-z-stab}}
    \label{fig:enter-label}
\end{figure}
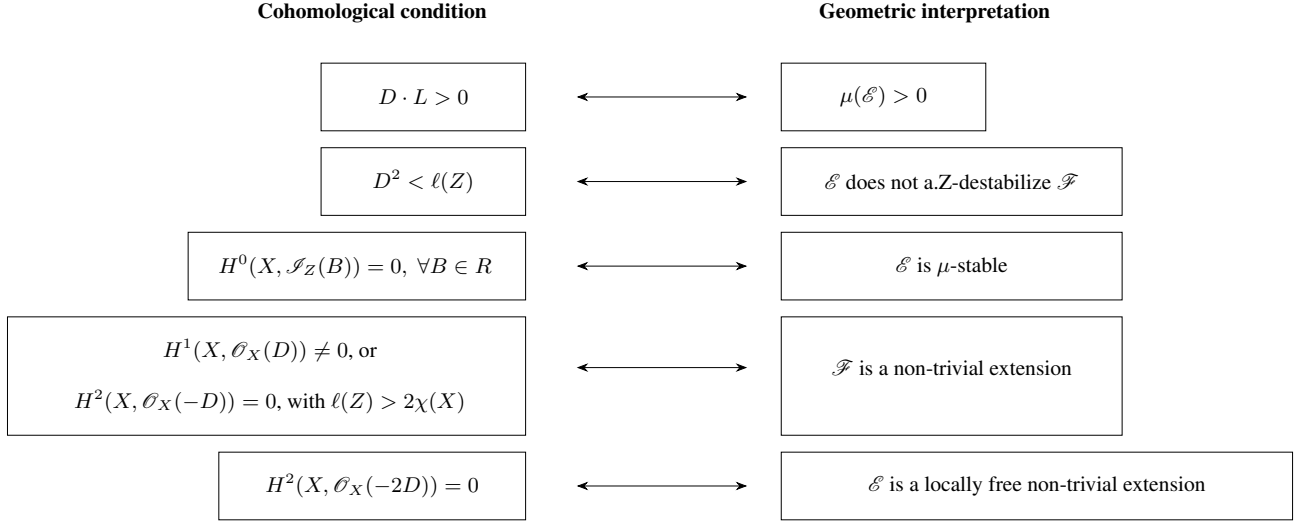

\end{remark}
Applying this construction, we obtain in \S\ref{sec:examples} some new examples of strictly a.Z-stable bundles over surfaces: 
\begin{enumerate}
    \item On $X=\P^2$, we have $\Pic(X) = \Z H$, where $H$ is the hyperplane class. In this case $R_0 = \{mH\in\Z\,|\,m\geq 0\}$ and we take the polarization $L = H$. Choosing $D = kH$, we have  \[R = \{m\in\Z\;|\;0\leq m \leq k\}.\]
    The conditions of Theorem~\ref{teo:hom-cond-for-stric-z-stab} are then satisfied for $k=1$ when $Z$ is the generic intersection of two curves in $\P^2$ with degrees at least $2$. Still on $\P^2$, we also obtain a direct example from the same scheme $Z$, taking $D = 2H$, but it does not hinge on  Theorem~\ref{teo:hom-cond-for-stric-z-stab}.
    \item On $X = \P^1\times \P^1$, we have $\Pic(X) = \Z A\oplus \Z B$, where  $A$ and $B$ are the classes of each copy of $\P^1\subset X$. In this case $R_0\subset \Z^2$ is the first quadrant and we take the polarization $L = A+B$. Choosing $D = -kA+\ell B$, for $k,\ell\in \Z$, we have
    \begin{equation*}
        R = \{aA+bB\;|\; a,b\geq 0\text{ and }a+b \leq \ell-k\}. 
    \end{equation*}
    The conditions of Theorem~\ref{teo:hom-cond-for-stric-z-stab} are then satisfied, if $\ell > k > 0$, when $Z=C_1\cap C_2$ is the intersection of curves with respective bi-degree $(a_i,b_i)$ such that $a_i,b_i > \ell$.
    \item On $X = {\rm Bl}_q\P^2$, we have $\Pic(X) = \Z H\oplus \Z E$, where $H$ is the hyperplane class and $E$ is the exceptional divisor. In this case $R_0=\{aH+ bE \;|\; a\geq 0\text{ and }a+b\geq 0\}$ (see Lemma~\ref{lemm: blp2-null-cohom}) and we take the polarization $L = 3H-E$. First, choosing $D = -H+4E$, we have
    \[R = \{0, E\}.\] 
    The conditions of Theorem~\ref{teo:hom-cond-for-stric-z-stab} are then satisfied when $Z$ is the (pullback of) intersection of two lines on $\P^2$ that do not pass through the blown-up point $q$. Second, choosing $D = -H+5E$, we have
    \[R = \{0, E, H-E\}.\]
    The conditions of Theorem~\ref{teo:hom-cond-for-stric-z-stab} are then satisfied when $Z$ is the intersection of $C_1 = 2H$ and $C_2 = 2H-2E$.  
\end{enumerate}

\setcounter{theorem}{0}
\renewcommand{\thetheorem}{\thesection.\arabic{theorem}}

\bigskip

\paragraph{\textbf{Acknowledgements}} The authors are grateful to Marcos Jardim and Éder Correa for proposing this problem and contributing with fruitful subsequent discussions. 
LL has been funded by the São Paulo Research Foundation (FAPESP), Brasil. Process Number \mbox{\#2020/15054-2}. 
HSE has been funded by the the Brazilian National Council for Scientific and Technological Development (CNPq)  \mbox{[307217/2017--5]} and the São Paulo Research Foundation (FAPESP), Brasil, Process Number  \mbox{\#2018/21391--1} and the BRIDGES: Brazil-France interplays in Gauge Theory, extremal structures and stability  collaboration (FAPESP-ANR), Brasil, Process Number \mbox{\#2021/04065--6}.

\section{Basic facts and definitions}
\label{sec:basic-facts}

Throughout this paper, $(X,L)$ denotes a polarized algebraic surface. 
We also assume that $X$ is  \emph{polycyclic}, in the sense that $\Pic(X) \cong \Z^r$, for some $r\in\N$.
The polarization $L$ makes $X$ a projective variety, and therefore, we have an isomorphism $\Pic(X) \cong \mathrm{Div}(X)/\sim$, defined modulo linear equivalence of divisors. We will denote by $\sO_X(D)$ the line bundle associated to the divisor $D$, and we will also write $D$ for its divisor class $[D]\in\mathrm{Div}(X)/\sim$. We will now describe several quantities derived from these objects. 

Suppose that $\sE\to X$ is a coherent sheaf and let $\alpha \coloneqq\frac{L^2}{2}$. For each $k\in\N$, the \emph{$k$-central charge} of $\sE$, mentioned in ~\eqref{eq:k-central-charge}, is defined by 
\begin{equation}
\label{eq:k-central-charge-2}
    Z_k(\sE) \coloneqq -k c_1(\sE)\cdot L + i\left(\alpha k^2\rk(\sE) - \ch_2(\sE)\right),
\end{equation}
and the \emph{$k$-slope} of $\sE$ is
\begin{equation}
\label{eq:k-slope-2}
    \mu_k(\sE) \coloneqq \frac{kc_1(\sE)\cdot L}{\alpha k^2\rk(\sE) - \ch_2(\sE)}.
\end{equation}
The following result, proved in~\cite{dervan_z-critical_2024}*{Lemma~2.11},
relates the asymptotic $Z$-stability to $\mu$-stability. 
\begin{lemma}
    If $\sE$ is a a.Z-stable sheaf, then $\sE$ is $\mu$-semistable.
\end{lemma}
This allows us to characterize a.Z-stability on surfaces in a way which will be useful later in the proof of Proposition~\ref{prop:az-stability-gieseker}, when comparing it with Gieseker stability if $X$ is a Del Pezzo surface. The main feature of Del Pezzo surfaces is their natural anti-canonical  polarization, which makes it possible to compare stability conditions via the Riemann-Roch formula.
\begin{proposition}
\label{prop:az-stability-dim2}
    A coherent sheaf $\sE\to X$ is a.Z-stable if, and only if, for every proper torsion-free subsheaf $\sF\hookrightarrow \sE$, exactly one of the following conditions holds:
    \begin{enumerate}[i)]
        \item\label{it:1} $\mu(\sF) < \mu(\sE)$;
        \item\label{it:2} $\mu(\sF)=\mu(\sE)$ and $\ch_2(\sF)c_1(\sE)\cdot L < \ch_2(\sE)c_1(\sF)\cdot L.$
    \end{enumerate}
    In particular, if $\sE$ is a.Z stable and $c_1(\sE)\cdot L = 0$, then $\sE$ is $\mu$-stable.
\end{proposition}
\begin{proof}
    From the definition of the slope in \eqref{eq:k-slope-2},
    \begin{equation*}
        \mu_k(\sE)-\mu_k(\sF)=\alpha k^3 (\rk(\sF)c_1(\sE)\cdot L - \rk(\sE)c_1(\sF)\cdot L) + k(\ch_2(\sE)c_1(\sF) - \ch_2(\sF)c_1(\sE)\cdot L).
    \end{equation*}
    So if $\mu(\sF)<\mu(\sE)$, then $\rk(\sF)c_1(\sE)\cdot L - \rk(\sE)c_1(\sF)\cdot L>0$, and this implies $\mu_k(\sE) - \mu_k(\sF) > 0$ for  $k\gg 0$. On the other hand, if $\mu(\sE) = \mu(\sF)$, then $$\mu_k(\sE) - \mu_k(\sF) =k(\ch_2(\sE)c_1(\sF) - \ch_2(\sF)c_1(\sE)\cdot L),$$ which is positive if, and only if, $\ch_2(\sF)c_1(\sE)\cdot L < \ch_2(\sE)c_1(\sF)\cdot L$. Finally, if $\mu(\sE) < \mu(\sF)$, then $$\rk(\sF)c_1(\sE)\cdot L - \rk(\sE)c_1(\sF)\cdot L < 0,$$ and this implies $\mu_k(\sE) - \mu_k(\sF) < 0$ for $k\gg 0$, meaning that $\sF$ a.Z-destabilizes $\sE$.
    
    Now, suppose that $c_1(\sE)\cdot L = 0$ and $\sE$ is a.Z-stable. The first part of the Proposition shows that the only possible $\mu$-destabilizing subsheaf $\sF\subset\sE$ must satisfy $\mu(\sF) = \mu(\sE)$, but in that case $c_1(\sF)\cdot L =0$. Now, since $\sE$ is a.Z-stable and $\mu(\sE) = \mu(\sF) = 0$, from item~\ref{it:2}) we must have
    $$0<\ch_2(\sF)c_1(\sE)\cdot L < \ch_2(\sE)c_1(\sF)\cdot L = 0,
    $$ 
    which is a contradiction. Thus, there are no $\mu$-destabilizing subsheafs $\sF$ for $\sE$. 
\end{proof}

Now we would like to compare a.Z-stability with the concept of Gieseker stability, 
following its presentation  in~\cite{Friedman98}*{Section~2.4}.
\begin{definition}
    Let $\sE$ be a coherent sheaf over $X$ and let $\sL = \sO_X(L)$ be the line bundle defined by the polarization $L$. We say that $\sE$ is Gieseker stable, with respect to $L$, if for every proper coherent subsheaf $\sF\hookrightarrow \sE$ we have 
    \begin{align}
        \frac{\chi(\sF\otimes \sL^{k})}{\rk(\sF)} < \frac{\chi(\sE\otimes \sL^{k})}{\rk(\sE)}, \qforq k\gg 0.
    \end{align}
\end{definition}
\begin{proposition}\label{prop:gieseker-stability}
    A coherent sheaf $\sE$ is Gieseker stable if, and only if, for each coherent subsheaf $\sF\hookrightarrow \sE$ one of the following conditions hold:
    \begin{enumerate}[i)]
        \item $\mu(\sF) < \mu(\sE)$
        \item $\mu(\sF)=\mu(\sE)$ and $\frac{\chi(\sF)}{\rk(\sF)} < \frac{\chi(\sE)}{\rk(\sE)}$.
    \end{enumerate}
\end{proposition}
We also recall the following Riemann-Roch formula on surfaces: 
\begin{proposition}
\label{prop:Euler-characteristic}
    The Euler characteristic of a coherent sheaf $\sE\to X$ satisfies 
    \begin{equation}\label{eq:riemann-roch-dim2}
        \chi(\sE) = \ch_2(\sE) - \frac{c_1(\sE)\cdot K_X}{2} + \rk(\sE)\chi(\sO_X).
    \end{equation}
\end{proposition}

With this information, we have the relation below between Gieseker stability and asymptotic $Z$-stability. The next proposition is a direct consequence of the previous definitions and results.

\begin{proposition}
\label{prop:az-stability-gieseker}
    Suppose that $(X,L = -K_X)$ is a Del Pezzo surface. Then a coherent sheaf $\sE\to X$ with $c_1(\sE)\cdot K_X < 0$ is a.Z-stable if, and only if, it is Gieseker stable. 
\end{proposition}
\begin{proof}
    It follows from \eqref{eq:riemann-roch-dim2} that, for every coherent sheaf $\sE$, we have 
    \begin{equation*}
        \frac{\chi(\sE)}{\rk(\sE)} = \frac{\ch_2(\sE)}{\rk(\sE)} + \frac{1}{2}\mu(\sE) + \chi(\sO_X),
    \end{equation*}
    And thus, from Proposition~\ref{prop:gieseker-stability} it  follows that $\sE$ is Gieseker stable if, and only if, for every coherent subsheaf $\sF\hookrightarrow \sE$ we have either $\mu(\sF) < \mu(\sE)$ or $\mu(\sF) = \mu(\sE)$ and $\frac{\ch_2(\sF)}{\rk(\sF)} < \frac{\ch_2(\sE)}{\rk(\sE)}$. Comparing with Proposition~\ref{prop:az-stability-dim2} we see that $\mu(\sF)<\mu(\sE)$ is exactly item~\ref{it:1}), so we must show that, when $\mu(\sF) = \mu(\sE)$, the condition  $\frac{\ch_2(\sF)}{\rk(\sF)} < \frac{\ch_2(\sE)}{\rk(\sE)}$ is equivalent to the inequality on item~\ref{it:2}), which in this cases is
    \[ \ch_2(\sF)c_1(\sE)\cdot (-K_X) < \ch_2(\sE)c_1(\sF)\cdot (-K_X).
    \]
    Multiplying both sides by $\rk(\sE)$ and using again that $\mu(\sF)=\mu(\sE)$, the inequality is equivalent to
    \begin{equation*}
       -\rk(\sE)\ch_2(\sF)c_1(\sE)\cdot K_X < -\rk(\sF)\ch_2(\sE)c_1(\sE)\cdot K_X.
    \end{equation*}
     Since $- c_1(\sE)\cdot K_X > 0$, we can cancel this common factor on both sides preserving the inequality, obtaining
     \begin{gather*}
        \rk(\sE) \ch_2(\sF) < \rk(\sF) \ch_2(\sE). 
        \qedhere
     \end{gather*}
\end{proof}
\begin{remark}
\label{rem:az-stability-gieseker}
    Due to this relation between a.Z-stability and Gieseker stability, we expect to be able to adapt the techinques for constructing strictly Gieseker stable bundles to the case of a.Z-stability. In particular, we will construct a.Z-stable bundles on surfaces by using the same techniques as in~\cite{okonek_vector_1980}, in accordance with the following explicit motivation. It is shown in~\cite{okonek_vector_1980} that if $\sE\to \P^2$ is a $\mu$-stable bundle satisfying $c_1(\sE)=0$ and $H^1(\P^2,E) \neq 0$, then every holomorphic bundle $\sF$ arising as a an extension of the form
    \begin{equation*}
        \ses{\sE}{\sF}{\sO_{\P^2}}
    \end{equation*}
    is strictly Gieseker stable.  Since this construction gives bundles with $c_1(\sF)=0$, there is no hope of obtaining a a.Z-stable bundle, because of Proposition~\ref{prop:az-stability-dim2}. So we need to adapt it, in order to produce $c_1(\sF)\neq 0$. 
\end{remark}
\section{Adapted Hoppe criteria}\label{sec:adapted-hoppe-criteria}
The main goal of this section is to provide cohomological criteria ensuring the (semi-)stability of holomorphic bundles over polarized varieties in some special settings, i.e. choices of variety, bundle and stability condition. We begin recalling the version of Hoppe's criterion for $\mu$-stability of rank $2$ bundles over polycyclic varieties, cf.~\cite{Jardim2017}*{Corollary~4}, 
applied to Hartshorne-Serre bundles. These will be useful ingredients to construct a.Z-stable bundles later. We then present an adpated version of Hoppe's criterion for the a.Z-stability of vector bundles of low rank, see also~\cite{okonek_vector_1980}*{Lemma~1.2.5}. 

\subsection{Stability of Hartshorne-Serre bundles}\label{subsec:stability-HS-bundles}

In this section we will study the $\mu$-stability of Hartshorne-Serre sheaves, following the references~\cites{Arrondo2007,Jardim2017}. In particular, we will be interested in vector bundles appearing as extensions of the following form
\begin{equation}\label{eq:HS-bundle}
    \ses{\sO_X}{\sE}{\sI_Z(2D)}
\end{equation}
for some divisor $D\in\mathrm{Div}(X)$ and a $0$-dimensional  subscheme $Z\subset X$. This will be useful for our construction \S~\ref{sec:const-az-stab}. To do this, one useful tool is the following form of Hoppe's criterion, adapted for rank $2$ vector bundles over polycyclic varieties.

\begin{proposition}[\cite{Jardim2017}, Corollary~4] \label{prop:hoppe-crit}
    Let $\sE\to X$ be a holomorphic vector bundle of rank $2$ over polycyclic variety with dimension $n$, and let $L$ be a polarization on $X$. The  bundle $\sE$ is $\mu$-(semi)-stable if, and only if, 
    \begin{equation}\label{eq:hoppe-crit}
        H^0(X,\sE(B)) = 0 \quad \forall B\in\Pic(X)\;\text{ with }\; B\cdot L^{n-1} \underset{(\leq)}{<} -\mu(\sE).
    \end{equation}  
\end{proposition}
The following direct consequence of Proposition~\ref{prop:hoppe-crit} is adapted to our context. This is the same result of~\cite{Jardim2017}*{Proposition~13}, but written with a slightly different wording and notation. 
\begin{proposition}
\label{prop:HS-stability}
    Let $(X,L)$ be a polarized surface, and let $\sE\to X$ be a holomorphic bundle given by an extension as in~\eqref{eq:HS-bundle}. Let $R_0\subset \Pic(X)$ be such that $H^0(X,\sO_X(B)) = 0$ for all $B\notin R_0$. If $D\cdot L > 0$ and $H^0(X,\sI_Z(B)) = 0$, for all $B\in R_0$ such that $B\cdot L\leq D\cdot L$, then $\sE$ is $\mu$-stable.
\end{proposition}
\begin{proof}
    Let $B\in\Pic(X)$ be such that $B\cdot L \leq -\mu(\sE) = - D\cdot L$. Since $D\cdot L > 0$ and $L$ is ample, we conclude that $H^0(X,\sO_X(B)) = 0$. Now, twisting the extension~\eqref{eq:HS-bundle} by $\sO_X(B)$, we get the short exact sequence
    \begin{equation*}
        \ses{\sO_X(B)}{\sE(B)}{\sI_Z(2D + B)},
    \end{equation*}
    and from its associated long exact sequence in cohomology, together with the previous fact, we conclude that there exists an inclusion
    \begin{equation*}
        H^0(X,\sE(B))\hookrightarrow H^0(X,\sI_Z(2D + B)).
    \end{equation*}
    Hence, $\sE$ being $\mu$-stable amounts to the vanishing of the  cohomology on the right-hand side. To see this, notice that $B\cdot L\leq  - D\cdot L$ is equivalent to $(2D + B)\cdot L \leq D\cdot L$. If $2D+B\in R_0$, then $H^0(X,\sI_Z(2D + B)) = 0$ by hypothesis, and if $2D  +B\notin R_0$ we have     \begin{equation*}\belowdisplayskip=-18pt
        H^0(X,\sI_Z(2D + B)) \hookrightarrow H^0(X,\sO_X(2D + B)) = 0. 
    \end{equation*}\qedhere
\end{proof}
\begin{remark}
    In order to obtain concrete examples from the previous Lemma to get examples, one should choose judiciously the polarization, so that the set of divisors $B\in R_0$ such that $B\cdot L\leq D\cdot L$ is finite, thus the vanishing of global sections imposes a finite number of constraints on $Z$.   
\end{remark}

\subsection{Cohomological criterion for a.Z-stability}
\label{subsec:hoppe-az-stab}

We are now interested in the case of a.Z-stability. We prove a version of Hoppe's criterion for the a.Z-stability of low rank vector bundles. The result itself will not be used in the rest of the paper, but it could potentially be useful in the future study a.Z-stability over surfaces. The proof is a direct adaptation of  Lemma~1.2.5 in~\cite{okonek_vector_1980}. 
\begin{lemma}
\label{lemm: mu-k-under-exact-sequences}
    Let $\sE,\sF$ and $\mathscr{H}$ be coherent sheaves on $(X,L)$. Suppose we have a short exact sequence 
    \[\ses{\sF}{\sE}{\mathscr{H}}.\]
    Then  the following equivalence holds for all $k\gg 0$: 
    \begin{equation}
    \mu_k(\sF) < \mu_k(\sE) \iff \mu_k(\mathscr{H}) > \mu_k(\sE).
    \end{equation}
\end{lemma}
\begin{proof}
    For any coherent sheaf $\sE$, the leading coefficient of $\Im(Z_k(E))$ is $\rk(\sE)>0$, so
    \begin{equation*}
        \Im(Z_k(\sE)) > 0 \qandq \Im(Z_k(\sE)) - \Im(Z_k(\mathscr{H})) =  \Im(Z_k(\sF))  > 0,
        \qforq k \gg 0.
    \end{equation*}
    Then
    \belowdisplayskip=-18pt
        \begin{align*}
            \mu_k(\sF) < \mu_k(\sE) \iff& \Im(Z_k(\sE))\Re(Z_k(\sF)) > \Im(Z_k(\sF))\Re(Z_k(\sE)\\
            \iff& -\Re(Z_k(\mathscr{H}))\Im(Z_k(\sE)) > -\Re(Z_k(\sE))\Im(Z_k(\mathscr{H}))\\
            \iff&\mu_k(\mathscr{H}) > \mu_k(\sE).
        \end{align*}\qedhere
\end{proof}

We are now able to prove the adapted version of Hoppe's criterion. 

\begin{proposition}\label{prop:asymp-hopp}
    Let $\sE\to X$ be asymptotically $Z$-stable vector bundle. Then there exists a $k_0$ such that, for all $k\geq k_0$,    \begin{equation}\label{eq: cohomology-vanish-1}
        H^0(X, \sE(B)) = 0, \quad \forall B\in \Pic(X)
        \qwithq \mu_k(\sO(B)) \leq -\mu_k(\sE),
    \end{equation}
    and 
    \begin{equation}
    \label{eq: cohomology-vanish-2}
        H^0(X, \sE^\ast(B)) = 0, \quad \forall B\in \Pic(X)
        \qwithq \mu_k(\sO(B)) \leq \mu_k(\sE).
    \end{equation}
    Conversely, and if moreover $rk(\sE) = 2$, then equation~\eqref{eq: cohomology-vanish-1} imply that $\sE$ is asymptotically $Z$-stable.
\end{proposition}
\begin{proof}
    If $B\in \Pic(X)$ is such that $\mu_k(\sO_X(B))\leq -\mu_k(\sE)$ and $H^0(X,\sE(B)) \neq 0$, $\forall k>0$, then there exists an inclusion $\sO_X(-B)\hookrightarrow \sE$ which destabilizes $\sE$:
    \begin{align*}
        \mu_k(\sO_X(-B)) = - \mu_k(\sO_X(B)) \geq \mu_k(\sE), \quad \forall k>0.
    \end{align*}
    On the other hand, if $B\in\Pic(X)$ with $H^0(X,\sE^\ast(B))\neq 0$ satisfies $\mu_k(\sO_X(B))\leq \mu_k(\sE)$, $\forall k>0$, then we have an inclusion $\sO_X(-B)\hookrightarrow \sE^\ast$, which dualizes to an exact sequence  
    \[\ses{\sF}{\sE}{\sI_C(B)},\]
    where $C$ is some $0$-dimensional scheme and $\sF\subset \sE$ is the dual of the quotient $\sE^\ast/\sO_X(-B)$.
    Now we have
    \begin{equation}
        \mu_k(\sI_C(B)) = \frac{k B\cdot L}{k^2\alpha - \frac{B^2}{2} + \ell(C)} \leq \mu_k(\sO_X(B)) \leq \mu_k(\sE)
    \end{equation} 
    and Lemma~\ref{lemm: mu-k-under-exact-sequences} says that this last inequality is equivalent to $\mu_k(\sF) \geq \mu_k(\sE)$, for $k\gg 0$, hence $\sF$ destabilizes $\sE$. Therefore stability of $\sE$ implies equations~\eqref{eq: cohomology-vanish-1} and~\eqref{eq: cohomology-vanish-2}.

    For the converse, assume $\sE$ has rank $2$ and \eqref{eq: cohomology-vanish-1} holds for $k\gg 0$. If $\sF$ has rank $1$, then it has the form $\sF = \sO_X(B)$ for some $B\in\Pic(X)$, and the inclusion $\sO_X(B)\hookrightarrow \sE$ implies that $H^0(X,\sE(-B))\neq 0$. Using the assumptions, we get 
    \[-\mu_k(\sE) < \mu_k(\sO_X(-B)) = - \mu_k(\sO_X(B)) = - \mu_k(\sF),\]
    i.e. $\mu_k(\sF) < \mu_k(\sE)$. This covers the converse for $\rk(\sE) = 2$. 
    \end{proof}
\begin{remark}
    We notice that conditions~\eqref{eq: cohomology-vanish-1} and~\eqref{eq: cohomology-vanish-2} are equivalent if $\rk(\sE) = 2$.
\end{remark}
    
\section{Constructing a.Z-stable bundles over surfaces}
\label{sec:const-az-stab}  

We present a method for constructing rank $3$ a.$Z$-stable bundles over a surface $X$. We start by constructing a rank $2$ $\mu$-stable bundle $\sE\to X$, and use it to obtain a rank $3$ bundle $\sG$ from an extension of $\sO_X(2D)$ by $\sE$. The sufficient condition for $\sG$ to be a.Z-stable is the content of Proposition~\ref{teo:strict-z-stable-bun}. The main insight for the construction is the example of a strictly Gieseker stable bundle over $\P^2$ presented in Remark~\ref{rem:az-stability-gieseker}.
    \begin{proposition}
    \label{teo:strict-z-stable-bun}
            Let $(X,L)$ be a polarized and polycyclic complex surface. Let $\sE\to X$ be a $\mu$-stable holomorphic bundle with $\rk(\sE)=2$ and $c_1(\sE) = 2D\in \Pic(X)$, and let $\sG$ the holomorphic bundle obtained by an extension
            \begin{equation}\label{eq:def-of-F}
                \ses{\sE}{\sG}{\sO_X(D)}.  
            \end{equation}
            Then $\sE$ is the only possible a.$Z$-destabilizer for $\sG$. 
            
            Furthermore, if  $D\cdot L> 0$ then the bundle $\sG$ is asymptotically Z-stable if, and only if, 
            \[2\ch_2(\sO_X(D))-\ch_2(\sE) > 0,\]
            or, equivalently,
            \begin{equation}
                c_2(\sE) > D^2.  
            \end{equation}
    \end{proposition}
    \begin{proof}
        Let $\sF\subset \sG$ be a subsheaf such that the quotient $\sG/\sF$ is torsion-free, and suppose that $\sF$ is a a.Z-destabilizer for $\sG$. Since $\rk(\sF)<3$, we divide the proof according to rank: 
        \begin{description}
            \item[$\rk(\sF) = 1$] Here we have $\sF = \sO_X(B)$ for some $B\in\Pic(X)$. For $k\geq 0$ we have 
    	\begin{align}
    		0 \geq \mu_k(\sG)-\mu_k(\sF) = (D\cdot L - B\cdot L)k^3 + O(k^2) \implies B\cdot L \geq D\cdot L,
    	\end{align}
            which implies that either $\Hom(\sO_X(B),\sO_X(D))=0$ or $D\cdot L = B\cdot L$. In the former case, the diagram 
            \begin{equation}
            \begin{tikzcd}
                &\sO_X(B)\arrow[d,hook]&\\
                \sE\arrow[r,hook]&\sF\arrow[r,two heads]&\sO_X(D)
            \end{tikzcd} 
            \end{equation}
            leads to an inclusion $\sO_X(B)\hookrightarrow \sE = \ker(\sG\twoheadrightarrow \sO_X(D))$, contradicting the $\mu$-stability of $\sE$. In the latter case, $\mu$-stability of $\sE$ implies $H^0(X,\sE(-B)) = 0$, and then the exact sequence 
            \begin{equation}
                \ses{\sE(-B)}{\sG(-B)}{\sO_X(D-B)}
            \end{equation}
            leads, after taking global sections, to the sequence
            \[H^0(X,\sG(-B))\hookrightarrow H^0(X,\sO_X(D-B))\overset{\delta}{\to} H^1(X,\sE(-B)).\]
            Here we have $H^0(X,\sO_X(D-B))=\C$ and the map $\delta$ is non-trivial. Therefore, $H^0(X,\sG(-B)) = \ker\delta = 0$, contradicting $\sO_X(B)\hookrightarrow \sG$.
            \medskip
            \item[$\rk(\sF) = 2$] Let $c_1(\sF) = B\in\Pic(X)$, and denote $\sQ = \sG/\sF$. The dual $\sQ^\ast$ is reflexive so it is locally-free, then $\sQ^\ast = \sO_X(B-3D)$, and again
            \begin{align}
    		0 \geq \mu_k(\sG)-\mu_k(\sF)\implies& B\cdot L \geq 2D\cdot L\\
            \implies&(3D-B)\cdot L\leq - D\cdot L = \mu(\sE^\ast). \label{eq:1}
    	\end{align}
            Since $\sE^\ast$ is $\mu$-stable, Proposition~\ref{prop:hoppe-crit} implies $0 = H^0(X,\sE^\ast(3D-B)) = \Hom(\sQ^\ast, \sE^\ast)$. Therefore, the diagram
            \begin{equation}
            \begin{tikzcd}
                &\sO_X(3D-B)\arrow[d,hook]&\\
                \sO(-D)\arrow[r,hook]&\sG^\ast\arrow[r,two heads]&\sE^\ast
            \end{tikzcd} 
            \end{equation}
            leads to an inclusion $\sQ^\ast\hookrightarrow \sO_X(-D)$, which is an isomorphism because of~\eqref{eq:1}, seeing as $\sQ^\ast$ is a line bundle. We now have the exact sequence $\ses{\sQ^\ast}{\sG^\ast}{\sE^\ast}$, which we dualize and fit into the following diagram
            \begin{equation*}
            \begin{tikzcd}
            \sF\arrow[hook]{r}\arrow{d} &\sG\arrow[two heads]{r} \arrow[hook]{d} &\sQ\arrow[hook]{d}\\
            \sE^{\ast\ast} \arrow[hook]{r} &\sG^{\ast\ast}\arrow[two heads]{r} &\sQ^{\ast\ast}  
            \end{tikzcd}.
        \end{equation*}
        Using the `snake lemma' and the fact that $\sG$ is reflexive, we conclude that $\sF \cong \sE^{\ast\ast} \cong \sE$. This shows that the only possible rank $2$ destabilizer of $\sG$ is $\sE$ itself.         
        Finally, the difference $\mu_k(\sG) - \mu_k(\sE)$ is a positive multiple of
    	\begin{equation} \label{eq:ch2-ineq}
    	   (D\cdot H)(2\ch_2(\sO_X(D)) - \ch_2(\sE)) > 0,
    	\end{equation}
    	and thus $\sE$ does not destabilize $\sG$.\qedhere
    \end{description}
\end{proof}
\begin{remark}
\label{rmk:bogomolov}
    If $\sF$ is obtained as an extension~\eqref{eq:def-of-F}, and we assume that $D\cdot L<0$, then
        \[(D\cdot H)(2\ch_2(\sO_X(D)) - \ch_2(\sE)) > 0 \iff c_2(\sE) < D^2,\]
        and following the proof of Proposition~\ref{teo:strict-z-stable-bun} we conclude that $\sF$ is a.Z-stable if, and only if, $c_2(\sE) < D^2$. On the other hand, we are assuming that $\sE$ is a $\mu$-stable vector bundle, so we can apply the Bogomolov inequality, cf.~\cite{Friedman98}*{Ch.~9}, to obtain
		\[c_1(\sE)^2 \leq 4c_2(\sE) \iff c_2(\sE)\geq D^2.\]
		This implies that $\sF$ can not be a.Z-stable.  
	\end{remark}
    The idea behind Theorem~\ref{teo:hom-cond-for-stric-z-stab} is to start with a rank $2$ bundle $\sE$ obtained via the Harthsorne-Serre construction from some suitable pair $(D,Z)$, then ensure its $\mu$-stability by the Hoppe-type criterion, and finally use it to extend $\sO_X(D)$, thus obtaining a strictly a.Z-stable bundle of rank $3$. The twisting of the ideal sheaf $\sI_Z$ by $2D$ is deliberate, to make sure that $c_1(\sE_{Z,D}) = 2D$. 
    One should also fine-tune $Z$ and $D$ in such a way that $\sE$ is non-trivial, locally-free and $\mu$-stable. It is known that this construction gives rise to rank $2$ stable vector bundles if the length $\ell(Z)$ is large enough, cf.~\cite{huybrechts_geometry_2010}*{Theorem~5.1.3}, but since this suitable $\ell(Z)$ depends on the dimension of the Hilbert scheme of divisors on $X$ with degree less than $\mu(\sE_{Z,D})$, we prefer a more concrete approach to finding good $D$ and $Z$. Since $\Pic(X) = \Z^k$,  the inequality $B\cdot L\leq a \in \R$ defines a semi-space in $\Z^k$ bounded by $a\in \R$, we call it $S_{a}$. If $D$ is such that $D\cdot L >0$, Proposition~\ref{prop:HS-stability} says that    \begin{equation}\label{eq:cond-on-z}
        H^0(X,\sI_Z(A)) = 0, \quad \forall A\in R \coloneqq R_0\cap S_{\mu(\sE)}. 
    \end{equation}
    is enough to guarantee the $\mu$-stability of $\sE$. If the set $R$ is finite, we can determine the 0-dimensional subscheme $Z\subset X$ as intersection of curves on $X$, satisfying a finite number of restrictions. Finally, the following Lemma, which is a direct consequence of~\cite{huybrechts_geometry_2010}*{Theorem~5.1.1}, deals with local freeness.    
    \begin{lemma}
    \label{lemm:non-triv-ext-rk-2}
        Suppose that $\sL$ is a line bundle over a projective surface $X$. If 
        \begin{equation*}
            H^2(X,\sL^\ast) = 0,
        \end{equation*} then for every 0-dimensional subscheme $Z\subset X$, there exists a non-trivial, locally-free extension of the form
        \begin{equation*}
            \ses{\sO_X}{\sE}{\sI_Z\otimes\sL}.
        \end{equation*}
    \end{lemma} 
Let us now prove this paper's main result.
\begin{proof}[Proof of Theorem~\ref{teo:hom-cond-for-stric-z-stab}]
    We start by constructing the rank $2$ $\mu$-stable bundle $\sE_{Z,D}$ over $X$, given by the extension 
    \begin{equation}\label{eq:HS-bun}
        \ses{\sO_X}{\sE_{Z,D}}{\sI_Z(2D)}.
    \end{equation}
    Since condition~\eqref{cond:5} is satisfied, we know by Lemma~\ref{lemm:non-triv-ext-rk-2} that $\sE_{Z,D}$ is non-trivial and locally-free. To see that $\sE_{Z,D}$ is $\mu$-stable, we just use condition~\eqref{cond:3} together with Proposition~\ref{prop:HS-stability}, then $\sE_{Z,D}$ is a rank $2$, indecomposable, $\mu$-stable holomorphic bundle over $(X,L)$. Now we apply Lemma~\ref{teo:strict-z-stable-bun} to construct the rank $3$ bundle $\sF_{Z,D}$, which is given by the extension
    \begin{equation*}
        \ses{\sE_{Z,D}}{\sF_{Z,D}}{\sO_X(D)}.
    \end{equation*}
    To ensure that it is non-trivial, we must compute $H^1(X,\sE_{Z,D}(-D))$. From the exact sequence~\ref{eq:HS-bun}, twisted by $-D$, we get the exact sequence
    \begin{equation*}
        H^0(X,\sI_Z(D))\to H^1(X, \sO_X(-D))\to H^1(X,\sE_{Z,D}(-D)),
    \end{equation*}
    and applying condition~\eqref{cond:3} again we get that $H^1(X,\sI_{Z}(D)) = 0$. This means that 
    \begin{equation*}
        0 \neq H^1(X, \sO_X(-D))\subset H^1(X,\sE_{Z,D}(-D)),
    \end{equation*}
    by the first part of condition~\eqref{cond:4}, which deals with non-triviality of the extension. Furthermore, if $H^1(X,\sO(-D)) = 0$ but $H^2(X,\sO(-D)) = 0$, then the same long exact sequence above will give us 
    \begin{equation*}
        H^1(X,\sE_{Z,D}(-D))\cong H^1(X,\sI_Z(D)).
    \end{equation*}
    In order to prove that this is non-zero, we use the twisted restriction sequence
    \begin{equation*}
        \ses{\sI_Z(D)}{\sO_X(D)}{\sO_Z}
    \end{equation*}
    to obtain
    \begin{equation*}
        h^1\sI_Z(D) = \ell(Z) + h^2\sO_X(D) - \chi(\sO_X(D)) \geq \ell(Z) - \chi(\sO_X) > 0,
    \end{equation*}
    by the second part of condition~\eqref{cond:4}.
    Finally, since $c_2(\sE_{Z,D}) = \ell(Z)$, the condition~\eqref{cond:2} together with Lemma~\ref{teo:strict-z-stable-bun} gives us that $\sF_{Z,D}$ is a.Z-stable.
\end{proof}

\newpage
    \section{Examples}
    \label{sec:examples}
       For each computation in this section, the base variety on which we compute cohomology groups of sheaves will be clear from the context. For notational simplicity, we let $h^p\sE\coloneqq H^p(X,\sE)$. 
   \subsection{The complex projective plane \texorpdfstring{$\P^2$}{}}
    For $X = \P^2$, we adopt the natural polarization $L=H\in \Pic(X)$ given by the hyperplane class, and accordingly also the most common notation for line bundles  $\sO_{\P^2}(k) \coloneqq \sO_{\P^2}(kH)$. The following proposition gives us a family of examples of strictly asymptotically $Z$-stable bundles over $\P^2$. 
    \begin{proposition}
    \label{teo:p2-str-stab-example}
    Let $X = \P^2$ with the polarization given by the hyperplane class. 
        \begin{enumerate}
            \item If $Z$ is the intersection of two curves with degrees $d_1$ and $d_2$, with $d_i \geq 2$, then there exists a non-trivial, locally-free sheaf $\sE_{Z,2}$ which is given by an extension
            \begin{align}
                \ses{\sO_X}{\sE_{Z,1}}{\sI_Z(2).}
            \end{align}
            Furthermore, this sheaf is $\mu$-stable.
            
            \item There exists a non-trivial locally free sheaf $\sF$ given by an extension of the form
            \begin{align}
                \ses{\sE_{Z,2}}{\sF_{Z,1}}{\sO_X(1)}.
            \end{align}
            Furthermore, this sheaf is strictly asymptotically $Z$-stable.
        \end{enumerate}
    \end{proposition}
    \begin{proof}
    We just have to check the conditions of Theorem~\ref{teo:hom-cond-for-stric-z-stab}.
    In that notation, we are choosing $D = L$, so that $D\cdot L = 1$, which gives us condition~\eqref{cond:1}. 
    We are taking moreover $Z$ to be the generic intersection of curves $C_1$ and $C_2$ in $\P^2$, with degrees at least 2. By Bézout's theorem, we have $\ell(Z)\geq 4 > 1 = D^2$, so condition~\eqref{cond:2} is also verified. 
    To check condition~\eqref{cond:3}, we recall the following exact sequence for the ideal sheaf $\sI_Z$:
    \begin{equation}
    \label{eq:int-exact-seq-p2}
        \ses{\sO_{\P^2}(-d_1 - d_2)}{\sO_{\P^2}(-d_1)\oplus \sO_{\P^2}(-d_2)}{\sI_Z},
    \end{equation}
    where $d_i$ is the degree of $C_i$. 
    We can take $R_0 = \{kL\;|\;k\geq 0\}$ and, since $D\cdot L = 1$, we have $R = \{0, L\}$, and one must check that $H^0(\P^2, \sI_Z) = 0 = H^0(\P^2,\sI_Z(1))$. The vanishing of $H^0(\P^2,\sI_Z)$ is clear since the restriction map $H^0(\P^2,\sO_{\P^2})\to H^0(\P^2,\sO_{Z})$ is injective. 
    On the other hand, twisting the exact sequence~\eqref{eq:int-exact-seq-p2} by $\sO_{\P^2}(1)$ and taking the global section functor, we get 
    \begin{equation*}
        H^0(\P^2, \sI_Z(1)) \subset H^1(\P^2, \sO_{\P^2}(-d_1 - d_2 + 1)) = 0,
    \end{equation*}
    since $d_i\geq 2$ implies that $H^0(X,\sO(-d_i + 1)) = 0$. Therefore, our setting satisfies condition~\eqref{cond:3}. 
    Now, we have
    \begin{equation*}
        H^2(\P^2,\sO_{\P^2}(-1)) = H^0(\P^2,\sO_{\P^2}(-2)) = 0, 
    \end{equation*}
    by Serre duality, and since $2\chi(\sO_{\P^2}) = 2 < 4 \leq \ell(Z)$ we also satisfy condition~\eqref{cond:4}.
        Finally, we get condition~\eqref{cond:5} by computing directly 
    \begin{equation*}
        H^2(\P^2,\sO_{\P^2}(-2)) = H^0(\P^2,\sO_{\P^2}(-1)) = 0. 
    \end{equation*}
    This concludes the proof.  
    \end{proof}

\begin{remark}
    The choice of $D=L$ was made in order to
    to satisfy condition~\eqref{cond:5}. Indeed, for $D = kL$, with $k> 0$,  we have
    \begin{equation*}
        H^2(\P^2, \sO_{\P^2}(-2k)) 
        = H^0(\P^2, \sO_{\P^2}(2k-3)),
    \end{equation*} 
    by Serre duality and the right-hand side vanishes if, and only if, $k = 1$. It is possible to get holomorphic bundles arising from an extension of the form
    \begin{equation}
    \label{eq:ext-p2-gen}
        \ses{\sO_{\P^2}}{\sE}{\sI_Z(2k)}, 
    \end{equation}
    just by taking $Z$ such that  $\ell(Z) > h^0(\P^2,\sO_{\P^2}(2k-3))$, see~\cite{huybrechts_geometry_2010}*{p.~148}. We can also examine the stability of the holomorphic bundle $\sE$ defined by the exact sequence~\eqref{eq:ext-p2-gen}: since $D\cdot L = k$, in this case we have $R = \{aL\;|\;0\leq a \leq k\}$, so one must verify $H^0(\P^2,\sI_Z(a)) = 0$, for all $0\leq a\leq k$. Again, assuming that $Z$ is a generic intersection $C_1 \cap C_2$, with respective degrees $d_1$ and $d_2$, we can twist the exact sequence~\eqref{eq:int-exact-seq-p2} by $\sO_{\P^2}(a)$ and apply the global section functor to obtain the projection
        \begin{equation*}
            H^0(\P^2,\sO_{\P^2}(a-d_1))\oplus  H^0(\P^2,\sO_{\P^2}(a-d_2))\twoheadrightarrow H^0(\P^2,\sI_Z(a)),
        \end{equation*} 
        and then the desired cohomology vanishes, for all $0\leq a\leq k$, if we take $d_i \geq k+1$. 

        On the other hand, in order to obtain a non-trivial holomorphic bundle arising as the extension
        \begin{equation*}
            \ses{\sE}{\sF}{\sO_{\P^2}(k)},
        \end{equation*}
        we would like to ensure that $H^2(\P^2,\sO_{\P^2}(-k)) = H^0(\P^2,\sO(k-3))$ vanishes, which happens if, and only if, $k=1,2$. 
        We already considered the case $k=1$, so we proceed to understand the constraints for the case $k = 2$. In this case $D^2 = 4 < 9 \leq \ell(Z)$, so the bundle $\sF$ is a.$Z$-stable. The existence of a locally-free extension is ensured by the fact that
        \begin{equation*}
            h^0\sO_{\P^2}(2k-3) = h^0\sO_{\P^2}(1) = 3 < 9 \leq \ell(Z)
        \end{equation*}
        We conclude that the diagram
    \begin{equation}
    \label{eq:ex-in-P2}
        \begin{tikzcd}[row sep=0.4cm, column sep=0.4cm]
            & \sO_{\P^2} \arrow[hook]{r} & \sE_{Z,2} \arrow[two heads]{r} \arrow[hook]{d}& \sI_Z(4) \\
            & & \sF_{Z,2} \arrow[two heads]{d} \\
            & & \sO_{\P^2}(2),
        \end{tikzcd}        
    \end{equation}
    also gives an example of a strictly a.$Z$-stable bundle which is not contemplated by Theorem~\ref{teo:hom-cond-for-stric-z-stab}. 
    \end{remark}

    \subsection{The product \texorpdfstring{$\P^1\times \P^1$}{}}
    Looking henceforth into polycyclic varieties, we begin with the product $X = \P^1\times\P^1$ of two projective lines. The Picard group is $\Pic(X) = \Z A \oplus \Z B$, where $A$ and $B$ are the generators corresponding to each copy of $\P^1$. The intersection matrix is given by $A^2 = B^2 = 1$ and $A\cdot B = 0$. As before, we shall use the notation $\sO_X(m,n)\coloneqq\sO_X(mA+nB)$. The following proposition gives us a family of examples of strictly asymptotically $Z$-stable bundles over $\P^1\times\P^1$. 
    \begin{proposition}\label{teo:quart-str-stab-example}
        Let $X=\P^1\times\P^1$, write $\Pic(X) = \Z A \oplus\Z B$ where $A$ and $B$ are the generators corresponding to each copy of $\P^1$, and fix the polarization $L=A+B$. The following assertions hold:
    \begin{enumerate}
        \item For every $\ell>k>0$, if $Z$ is the intersection of two curves $C_1$ and $C_2$, with bi-degrees $(a_i,b_i)$ such that $a_i,b_i > \ell$, then there exists a non-trivial holomorphic vector bundle $\sE_{Z, (-k,\ell)}$, which arises as an extension
        \begin{align}
            \ses{\sO_X}{\sE_{Z,(-k,\ell)}}{\sI_Z(-2k, 2\ell).}
        \end{align}
        Furthermore, $\sE_{Z, (-k,\ell)}$ is $L$-slope-stable.
        \item There exists a non-trivial holomorphic vector bundle $\sF_{Z,(-k,\ell)}$, which arises as an extension of the form
        \begin{align}
            \ses{\sE_{-2k,2\ell}}{\sF_{Z,(-k,\ell)}}{\sO_X(-k,\ell)}.
        \end{align}
        Furthermore, $\sF_{Z,(-k,\ell)}$ is strictly asymptotically $Z$-stable with respect to the polarization $L$.
    \end{enumerate}
    \end{proposition}
    \begin{proof}
    In the notation of Theorem~\ref{teo:hom-cond-for-stric-z-stab}, we choose, for each $\ell>k>0$, the divisor $D = -kA + \ell B$ and the subscheme $Z$ to be a generic intersection of two curves $C_1$ and $C_2$, with bi-degree $(a_i,b_i)$ such that $a_i,b_i > \ell$. 
    Before we get into the hypotheses of the construction Theorem, let us keep in mind the following  convenient table with the cohomologies of line bundles over $X$, denoting   $\sO_X(a,b):= \sO_X(aA+bB)$:
    \begin{table}[h!]
    \label{tab:cohomology-P1P1}
        \centering
        \begin{tabular}{|c|c|c|c|}
            \hline
            & $h^0(X, \sO_X(a,b))$ & $h^1(X, \sO_X(a,b))$ & $h^2(X, \sO_X(a,b))$ \\
            \hline
            $a \geq -1, b \geq -1$ & $(a+1)(b+1)$ & $0$ & $0$ \\
            \hline
            $a \geq -1, b \leq -1$ & $0$ & $(a+1)(-b-1)$ & $0$ \\
            \hline
            $a \leq -1, b \geq -1$ & $0$ & $(-a-1)(b+1)$ & $0$ \\
            \hline
            $a \leq -1 , b \leq -1$ & $0$ & $0$ & $(a+1)(b+1)$ \\
            \hline
        \end{tabular}
        \caption{Cohomology dimensions for $\sO_X(a,b)$ over $X = \P^1 \times \P^1$}
        \label{tab:cohomology}
    \end{table}
    Now, condition~\eqref{cond:1} is clearly satisfied since $D\cdot L = \ell - k > 0$.
    To verify condition~\eqref{cond:2}, we will need to compute $\ell(Z)$. First, considering the restriction exact sequence
    \begin{equation}
    \label{eq:rest-ex-seq}
        \ses{\sI_Z}{\sO_X}{\sO_Z},
    \end{equation}
    we apply the global section functor and use the fact that $H^1(X,\sO_X) = 0 = H^0(X,\sI_Z)$ to obtain 
    \begin{equation}
    \label{eq:card-z-P1P1}
        \ell(Z) = h^0\sO_Z = h^0\sO_X + h^1\sI_Z = 1 + h^1\sI_Z.
    \end{equation}
    Since $Z$ is the intersection of the curves $C_i$, we have the exact sequence
    \begin{equation}
    \label{eq:int-ex-sec-P1P1}
        \ses{\sO_X(-a_1-a_2,-b_1-b_2)}{\sO_X(-a_1,-b_1)\oplus \sO_X(-a_2,-b_2)}{\sI_Z},
    \end{equation}
    from which we extract that 
    \begin{align}
    \label{eq:bezout-P1P1}
        h^1\sI_Z &= h^2\sO_X(-a_1-a_2,-b_1-b_2) - h^2\sO_X(-a_1,-b_1) -h^2\sO_X(-a_2,-b_2) \nonumber\\
        & = a_1b_2 + a_2b_1 - 1.
    \end{align}
    Combining~\eqref{eq:card-z-P1P1} and~\eqref{eq:bezout-P1P1} gives us
    \begin{equation*}
        \ell(Z) = a_1b_2 + a_2b_1 > 2\ell^2 > k^2 + \ell^2 = D^2,
    \end{equation*}
    so condition~\eqref{cond:2} is satisfied. 
    In order to verify \eqref{cond:3}, we take $R_0 = \{aA+bB\;|\; a, b\geq 0\}$. It is clear from the Künneth formula that $H^0(X,\sO_X(C)) = 0$, for all $C\notin R_0$. Therefore, since $D\cdot L = \ell - k$, 
    \begin{equation*}
        R = \{aA+bB\;|\; a,b\geq 0\text{ and }a+b \leq \ell-k\}, 
    \end{equation*}
    and it remains to show that $H^0(X,\sI_Z(a,b)) = 0$ for each $a,b\geq 0$ such that $a+b \leq \ell - k$. We twist the exact sequence~\eqref{eq:int-ex-sec-P1P1} by $aA+bB$, to obtain
    \begin{equation*}
        \ses{\sO_X(-a_1-a_2+a,-b_1-b_2+b)}{\sO_X(-a_1+a,-b_1+b)\oplus \sO_X(-a_2+a,-b_2+b)}{\sI_Z(a,b)}, 
    \end{equation*}
    and since $a, b\leq \ell - k < \ell < a_i, b_i$, we find the inclusion in cohomology
    \begin{equation*}
        H^0(X,\sI_Z(a,b)) \subset H^1(X,\sO_X(-a_1-a_2+a,-b_1-b_2+b)).
    \end{equation*}
    But the right-hand side cohomology vanishes, since $-a_1 - a_2 + a < - 2\ell + \ell < -1$ and analogously $-b_1 - b_2 + b < -1$, which gives us \eqref{cond:3}. Condition~\eqref{cond:4} is also satisfied, since
    \begin{equation*}
        h^1\sO_X(k,-\ell) = (k+1)(\ell-1) > k^2 - 1 > 0.
    \end{equation*}
    Finally, we get~\eqref{cond:5} from the fact that $h^2\sO_X(2k,-2\ell) = 0$, see Table~\ref{tab:cohomology}.
    \end{proof}
    
\subsection{Blow up of \texorpdfstring{$\P^2$}{} at a point}
    
Now we seek for examples over the variety $X = {\rm Bl}_q\P^2$, the blow-up of $\P^2$ at a point $q$. We write $\Pic(X) = \Z H \oplus \Z E$, where $H$ is the pullback of the hyperplane class of $\P^2$ and $E$ is the exceptional divisor. The intersection form for these generators is given by
    \[H^2 = 1\quad E^2 = -1\quad H\cdot E = 0,\]
and the canonical divisor is $K_X = -3H + E$.
    
Again, we use the notation $\sO_X(aH + bE) \coloneqq \sO_X(a,b)$. Before proceeding to the construction of a.Z-stable bundles, we present a Lemma that provides a suitable set $R_0$, in the notation of Theorem~\ref{teo:hom-cond-for-stric-z-stab}, for this setup. It is inspired by the computations in~\cite{zhao_moduli_2022}. 
    
	\begin{lemma}\label{lemm: blp2-null-cohom}
		Let $\sO_X(a,b)$ be a line bundle over $X = {\rm Bl}_q\P^2$. If either $a<0$ or $b<-a<0$, then $H^0(X,\sO_X(a,b)) = 0$.
	\end{lemma}
	\begin{proof}
		Suppose first that $a<0$, and take the restriction exact sequence for $H$:
        \begin{equation}
        \label{eq:rest-exact-seq-H}
            \ses{\sO_X(-1,0)}{\sO_X}{\sO_H}.
        \end{equation}
		Since the induced map $H^0(X,\sO_X)\to H^0(H,\sO_H)$ is injective, the kernel  $H^0(X,\sO_X(-1,0))$ vanishes. Twisting ~\eqref{eq:rest-exact-seq-H} by $aH$ gives us the exact sequence
        \begin{equation}
        \ses{\sO_X(a-1,0)}{\sO_X(a,0)}{\sO_H(a)},
        \end{equation}
		and since $H^0(H,\sO_H(a)) = 0$ we get $H^0(X,\sO_X(a-1,0)) = H^0(X,\sO_X(a,0))$, all of which vanish. 
        
        On the other hand, if we twist the restriction exact sequence for $E$ by the divisor $aH+bE$, with $b\geq 0$, we get the exact sequence
        \begin{equation}
        \label{eq:rest-exact-seq-E}
            \ses{\sO_X(a,b-1)}{\sO_X(a, b)}{\sO_E(-b)},
        \end{equation}
		and since $H^0(E,\sO_E(-b))=0$ we can conclude that
        \[
            H^0(X,\sO_X(a,b)) = H^0(X,\sO_X(a,b-1)) = \cdots = H^0(X,\sO_X(a, 0)) = 0. 
        \]
        Now, if $b<0$, we have the inclusion
        \[\sO_X(a,b)\hookrightarrow \sO_X(a,b+1)\]
        which shows that 
        \[H^0(X,\sO_X(a,b)) = H^0(X,\sO_X(a,b+1)) = \cdots = H^0(X,\sO_X(a,0)) = 0,\]
        and this proves the first part.
		Now suppose that $b<-a<0$. First, the restriction exact sequence for $E$ is
        \begin{equation}\label{eq:rest-seq-E}
            \ses{\sO_X(0,-1)}{\sO_X}{\sO_E}.
        \end{equation}
        Since the map $H^0(X,\sO_X)\to H^0(X,\sO_E)$ is surjective we have $H^0(X,\sO_X(0,-1)) = 0$. Also we can conclude that $H^0(X,\sO_X(0,b)) = 0$ for all $b<0$. Now we twist the restriction exact sequence for $A = H-E$ by $aH+bE$, to get the exact sequence 
		\[\ses{\sO_X(a-1,b+1)}{\sO_X(a,b)}{\sO_A(a+b)},\]
		and since $a+b < 0$ we see that $H^0(A,\sO_A(a+b)) = 0$ and hence $H^0(X,\sO_X(a,b)) = H^0(X,\sO_X(a-1,b+1))$. We do this step $a$ times and obtain that 
		$H^0(X,\sO_X(a,b))=H^0(X,\sO_X(0,b+a)) = 0.$ 
	\end{proof}
    
    We present our first example in the form of the following proposition.

    \begin{proposition}
    \label{teo:blp2-str-stab-example}
        Let $X = {\rm Bl}_q\P^2$, write $\Pic(X) = \Z H\oplus \Z E$, where $H$ is corresponds to the hyperplane class, and $E$ is the exceptional divisor, and fix the polarization $L=3H - E$. The following hold:
        \begin{enumerate}
            \item If $Z$ is the intersection of two curves $C_i$, linearly equivalent to $H$, then there exists a holomorphic vector bundle $\sE_{Z,(-1,4)}$ which is given by a non-trivial extension
            \begin{align}
                \ses{\sO_X}{\sE_{Z,(-1,4)}}{\sI_Z(-2H + 8E).}
            \end{align}
            Furthermore, this bundle is $L$-slope-stable.
                
            \item There exists a holomorphic vector bundle $\sF$ given by a non-trivial extension of the form
            \begin{align}
                \ses{\sE_{Z,(-1,4)}}{\sF_{Z,(-1.4)}}{\sO_X(-H+4E)}.
            \end{align}
            Furthermore, $\sF$ is strictly asymptotically $Z$-stable with respect to $L$.
        \end{enumerate}
    \end{proposition}
    \begin{proof}
        We show that the choice of $(Z,D)$ satisfies the conditions in  Theorem~\ref{teo:hom-cond-for-stric-z-stab}. Condition~\eqref{cond:1} is clearly satisfied since $D\cdot L = -3 + 4 = 1$. Condition~\eqref{cond:2} is also quickly verified since 
    \[D^2 = 1-16 = -15 < 0 \leq \ell(Z),\]
    independently of $Z$. Skipping condition~\eqref{cond:3} for a moment, for condition~\eqref{cond:4} we compute $H^1(X,\sO_X(1,-4))$ using the Riemann-Roch formula:
    \begin{align*}
        h^1(\sO_X(1,-4)) &= h^0(\sO_X(1,-4))+ h^0(-4, 3)-\frac{1}{2}(H-4E)(4H - 3E) - 1\\
        &= 3.
    \end{align*}
    Condition~\eqref{cond:5} is also direct since
    \begin{equation*}
        H^2(X,\sO_X(-2,8)) = H^0(X,\sO_X(-1,-7)) = 0.
    \end{equation*}
    Finally, we address condition~\eqref{cond:3}. By Lemma~\ref{lemm: blp2-null-cohom} we can take 
    \begin{equation}
        R_0 = \{aH + bE\;|\;a\geq 0\text{ and } a+b\geq 0\},  
    \end{equation}
    and since $D\cdot L = 1 > 0$, we have
    \begin{equation}
        R = \{aH + bE\;|\;a\geq 0\text{, } a+b\geq 0 \text{ and } 3a+b\leq 1\} = \{0, E\}.
    \end{equation}
    Since we already know that $H^0(X,\sI_Z) = 0$, we just need to show that $H^0(X,\sI_Z(0,1)) = 0$. Recall that $Z$ is the intersection of the pullbacks of two lines in $\P^2$ that generically do not pass through $q$, and since each line is represented by $H$ in the Picard group, we have the following exact sequence defining $\sI_Z$: 
    \begin{equation}
    \label{eq:int-ex-seq-blp2}
        \ses{\sO_X(-2,0)}{\sO_X(-1,0)^{\oplus 2}}{\sI_Z}.
    \end{equation}
    Twisting~\eqref{eq:int-ex-seq-blp2} by $E$, we get 
    \[\ses{\sO_X(-2,1)}{\sO_X(-1,1)^{\oplus 2}}{\sI_Z(0,1)}.\]
    Looking at the associated long exact sequence in cohomology, we obtain
    \[H^0(X,\sI_Z(0,1)) \hookrightarrow H^1(X,\sO_X(-2,1)),\]
    since $H^0(X,\sO_X(-1,1)) = 0$. Finally, by the Riemann-Roch formula,
    \[h^1(\sO_X(-2,1)) = -\frac{1}{2}(-2H + E)\cdot H - 1 = 0,\]
    and we conclude that $H^0(X,\sI_Z(0,1)) = 0$.
    \end{proof}
    We use the same approach for the second example, which is given by the following proposition.
    \begin{proposition}
    \label{teo:blp2-str-stab-example-2}
        Let $X = {\rm Bl}_q\P^2$, write $\Pic(X) = \Z H\oplus \Z E$, where $H$ is corresponds to the hyperplane class and $E$ is the exceptional divisor, and fix the polarization $L=3H-E$. 
        \begin{enumerate}
            \item If $Z$ is the generic intersection of two curves $C_i$, linearly equivalent to $2H$ and $2H-2C$ respectively, then there exists a holomorphic vector bundle $\sE_{Z,(-1,5)}$ given by a non-trivial extension
            \begin{align}
                \ses{\sO_X}{\sE_{Z,(-1,5)}}{\sI_Z(-2H + 10E).}
            \end{align}
            Furthermore, $\sE_{Z,(-1,5)}$ is $L$-slope-stable.
            
            \item There exists a holomorphic vector bundle $\sF$ given by a non-trivial extension of the form
            \begin{align}
                \ses{\sE_{Z,(-1,4)}}{\sF_{Z,(-1,5)}}{\sO_X(-H+5E)}.
            \end{align}
            Furthermore, $\sF$ is strictly asymptotically $Z$-stable with respect to $L$.
        \end{enumerate}
    \end{proposition}
    \begin{proof}
        We show that the choice of $(Z,D)$ satisfies the conditions in  Theorem~\ref{teo:hom-cond-for-stric-z-stab}. Condition~\eqref{cond:1} is clearly satisfied since $D\cdot L = -3 + 4 = 2$. Condition~\eqref{cond:2} is also quickly verified since 
        \[D^2 = 1-25 = -24 < 0 \leq \ell(Z),\]
        independently of $Z$. Skipping condition~\eqref{cond:3} for a moment, for condition~\eqref{cond:4} we compute $H^1(X,\sO_X(1,-5))$, using the Riemann-Roch formula:
        \begin{align*}
            h^1(\sO_X(1,-5)) &= h^0(\sO_X(1,-5))+ h^0(-4, 6)-\frac{1}{2}(H-5E)(4H - 6E) - 1\\
            &= 12.
        \end{align*}
        Condition~\eqref{cond:5} is also direct since
        \begin{equation*}
            H^2(X,\sO_X(-2,10)) = H^0(X,\sO_X(-1,-9)) = 0.
        \end{equation*}
        Finally, for condition~\eqref{cond:3}, we recall  Lemma~\ref{lemm: blp2-null-cohom} to take 
        \begin{equation}
            R_0 = \{aH + bE\;|\;a\geq 0\text{ and } a+b\geq 0\}.  
        \end{equation}
        Since $D\cdot L = 2$, we have
        \begin{equation}
            R = \{aH + bE\;|\;a\geq 0\text{, } a+b\geq 0 \text{ and } 3a+b\leq 2\} = \{0, E, H-E\}.
        \end{equation}
        We already know that $H^0(X,\sI_Z) = 0$, so we just need to show that 
        \[H^0(X,\sI_Z(0,1)) = H^0(X,\sI_Z(1,-1)) = 0.\]
        Recall that $Z$ is the intersection of$C_1\sim 2H$ and $C_2 \sim 2H - 2E$, thus we have the following exact sequence defining $\sI_Z$: 
        \begin{equation}
        \label{eq:int-ex-seq-blp2-2}
            \ses{\sO_X(-4,2)}{\sO_X(-2,0)\oplus \sO_X(-2,2)}{\sI_Z}.
        \end{equation}
        Twisting~\eqref{eq:int-ex-seq-blp2-2} by $E$, we get 
        \[\ses{\sO_X(-4,3)}{\sO_X(-2,1)\oplus \sO_X(-2,4)}{\sI_Z(0,1)}.\]
        Looking at the associated long exact sequence in cohomology, we obtain
        \[H^0(X,\sI_Z(0,1)) \hookrightarrow H^1(X,\sO_X(-4,3)),\]
        since $H^0(X,\sO_X(-2,1)) = 0 = H^0(X,\sO_X(-2,4))$. On the other hand, the Riemann-Roch formula gives 
        \[h^1(\sO_X(-4,3)) = -\frac{1}{2}(-4H + 3E)\cdot (-H+2E) - 1 = 0,\]
        and therefore $H^0(X,\sI_Z(0,1)) = 0$. 
        Now we twist~\eqref{eq:int-ex-seq-blp2-2} by $H-E$, obtaining 
        \[\ses{\sO_X(-3,1)}{\sO_X(-1,-1)\oplus \sO_X(-1,1)}{\sI_Z(1,-1)}.\]
        The associated long exact sequence gives the injection  
        \[H^0(X,\sI_Z(1,-1))\hookrightarrow H^1(X,\sO_X(-3,1)),\]
        since $H^0(X,\sO_X(-1,-1))=0=H^0(X,\sO_X(-1,1))$. On the other hand, the Riemann-Roch formula again gives
        
        \[h^1(\sO_X(-3,1)) =h^2\sO_X(-3,1))-\frac{1}{2}(-3H + E)\cdot 0 - 1 = h^0(X,\sO_X)-1 = 0,\]
        which concludes the proof.
    \end{proof}

\bibliographystyle{alpha}
\bibliography{Bibliografia-2023-05}

\end{document}